\documentclass[11pt,oneside,reqno]{amsart}
\usepackage{mathtools}
\usepackage[bookmarks=true]{hyperref}
\usepackage{graphicx}
\usepackage{color}
\usepackage[update,prepend]{epstopdf}
\usepackage{amsmath,amsfonts}
 \usepackage{soul}
\usepackage{times}
\usepackage{cancel}
\usepackage[ansinew]{inputenc}
\usepackage{latexsym}
\newtheorem{rem}{Remark}[section]
\newtheorem{propn}{Proposition}[section]
\newtheorem{cory}{Corollary}[section]
\newcommand{\re}[1]{(\ref{#1})}

\newtheorem{definition}{Definition}[section]
\newtheorem{lem}{Lemma}[section]
\newtheorem{thm}{Theorem}[section]
\numberwithin{equation}{section}

\def\openbox{$\sqcup\llap{$\sqcap$}$}
   
\let\qed=\endproof

\let\cal=\mathcal

 \let \Om=\Omega   \let \G=\Gamma  
\let \v=\varepsilon   \let\p=\partial
   
\let \var=\varphi 
\let\d=\delta
\let\mref=\bibitem
\let\bb=\mathbb

\let\re=\ref
\let\eqre=\eqref

\def\endoc{\end{document}}
\def\beq*{\begin{equation*}}
 \def\eneq*{\end{equation*}}
 
  \def\openbox{$\sqcup\llap{$\sqcap$}$}
   
   \let\qed=\endproof

\date{\today}

\title[The wave equation with supercritical nonlinear damping]{Energy decay estimates for the  wave equation with supercritical nonlinear damping }
\begin{document}

\bibliographystyle{plain}

\maketitle

\centerline{\scshape Alain Haraux}
\medskip
{\footnotesize
	\centerline{ Laboratoire Jacques-Louis Lions}
	\centerline{Sorbonne University, Pierre and Marie Curie Campus}
	\centerline{4 Pl. Jussieu, 75005 Paris, France}
}

\medskip

\centerline{\scshape Louis Tebou}
\medskip
{\footnotesize
	\centerline{  Department of Mathematics and Statistics}
	\centerline{Florida International University, Miami, FL 33199, USA}
}

 \begin{abstract} We consider  a damped wave equation in a bounded domain.  The damping is nonlinear and is homogeneous with degree $p-1$ with $p>2$. First, we show that the energy of the strong solution in the supercritical case decays as a negative power of $t$; the rate of decay is the same as in the subcritical or critical cases, provided that the space dimension does not exceed ten.  Next, relying on a new differential inequality, we show that if the initial displacement is further required to lie in $L^p$, then the  energy of the corresponding weak solution decays  logarithmically  in the supercritical case. Those new results complement those in the literature and open an important breach in the unknown land of super-critical damping mechanisms.
 
\end{abstract}\vskip.3cm

\tableofcontents
\newpage

\section{Introduction}\label{weqs}

 Let $\Om$ be a nonempty
bounded subset of ${\bb R}^N$ with boundary $\G$ of class
$C^2$. Let $c>0$ and $p > 2$ be constants. 
\par\noindent Consider the nonlinearly damped wave equation
\begin{equation}\label{e1}\left\{
\begin{array}{lll}y_{tt}-\Delta y+c |y_t|^{p-2}y_t=0\hbox{ in
}(0,\infty)\times\Om\cr y=0\hbox{ on }(0,\infty)\times\G\cr
y(0)=y^0,\quad y_t(0)=y^1\hbox{ in }\Omega,\end{array}\right.\end{equation}
and let us introduce the energy \begin{equation}\label{e6}E(t)={1\over 2}\int_\Omega\{\vert
 y_t(t,x)\vert^2+|\nabla y(t,x)|^2\}\,dx,\quad\forall
t\geq 0.\end{equation} Then a formal calculation shows that the energy $E$ is a nonincreasing function
of the time variable $t$ and its derivative satisfies
\begin{equation}\label{e7}E'(t)=-c\int_{\Om}| y_t(t,x)|^p\,dx,\quad\forall t\geq0.\end{equation} It has been known for a long time that the energy of all weak solutions tends to $0$ as $t\to \infty$. Actually, that property is valid  (cf. e.g. \cite{h1}) for the more general wave equation \begin{equation}\label{eg}\left\{
\begin{array}{lll}y_{tt}-\Delta y+g(y_t)=0\hbox{ in
}(0,\infty)\times\Om\cr y=0\hbox{ on }(0,\infty)\times\G\cr
y(0)=y^0,\quad y_t(0)=y^1\hbox{ in }\Omega,\end{array}\right.\end{equation} under the sole hypothesis that $g \in C({\bb R}) $ is non-decreasing with $ g(0) = 0$ and 
$$ 0 \not\in Int (g^{-1}(0)). $$ 

It is natural to ask about the rate of decay to $0$ of $E(t)$. This question has also been studied for quite a while. It was shown (cf.e.g. \cite{nak0, h1}) that for equation \eqref{e1},  if $(N-2)p\leq 2N$, then we have 
\begin{equation}\label{ene0}
E(t)\leq C \left(1+t\right)^{-{\frac{2}{p-2}}} \end{equation} where $C$ depends on the initial energy $E(0)$. This rate of decay is natural since it is the exact decay of all non trivial solutions of the ODE 
$$ u''+ \omega^2 u + c|u'|^{p-2} u' = 0.  $$ 
However, at the present time, it is not known whether the rate of decay obtained for equation \eqref{e1} is optimal for {\it{any}}  non trivial solution, even very regular. And before the present paper, nothing was known concerning the rate of the solutions of equation \eqref{e1} when $(N-2)p> 2N. $ In this paper, we show that under some additional regularity hypotheses on the initial state $(y^0, y^1)$, the energy of the solution has an explicit rate of decay. More precisely, if 
$(y^0,y^1)\in \left(H^2(\Om)\cap H_0^1(\Om)\right) \times    \left(H^1_0(\Om)\cap L^{2(p-1)}(\Om)\right)$, we obtain a decay as a negative power of $t$ even for $N>2$ and some $p> \frac{2N} {N-2}$, $\left(\text{actually all  }p> \frac{2N} {N-2} \text{ when }N \text{ lies in }\{3,4\}\right)$. And under the much less restrictive conditions $y^0\in H_0^1(\Om)\cap L^p(\Om)$ and $y^1\in L^2(\Om)$, we obtain a logarithmic decay of the energy for all $p> \frac{2N} {N-2}.$ \\

The methods of this paper are very flexible and can be applied to many dissipative second order evolution equations. In order for the reader to catch easily the essentials of our arguments, we have decided to state and prove the results only for the model equation \eqref{e1}, and we did not try to elaborate a general functional framework. Such a more exhaustive study will be done in a forthcoming paper. However, in the last section of the present paper, we give a few examples of other equations in order for the reader to understand the modifications which occur when generalizing the technique to different situations. 

\section{Some preliminary results}\label{prelim}
 Before stating our main results in the next section, it is useful to recall the various notions of solutions for equations \eqref {e1} and \eqref {eg}. 
 
\begin{definition}  A function $y = y(t, x)$ defined on $[0,\infty)\times \Om $ is called a {\bf strong  solution}  of \eqre{eg} 
if \begin{equation}\label{ee5}y\in L^{\infty}_{loc}([0,\infty); H^2(\Om))
)\cap{W}^{1, \infty}_{loc}([0,\infty);H^1_0(\Om))\cap{W}^{2, \infty} _{loc} ([0,\infty);L^2(\Om)) 
 \end{equation} and $y$ satisfies \eqre{eg} in the obvious sense.
\end{definition}

 \begin{propn}\label{ss}  Assume that either $\Om$ is convex, or $\Omega$ has a $C^2$ boundary. Let  $(y^0,y^1)\in \left(H^2(\Om)\cap H_0^1(\Om)\right) \times   H^1_0(\Om)$ with $g(y^1)\in L^2(\Om)$. Then  \eqre{eg} has a unique strong solution $y.$ 
 Furthermore, if we set $$F(t)=\int_\Omega\{|\nabla y_t(t,x)|^2+|\Delta y(t,x)|^2\}\,dx,\quad\forall t\geq0,$$then $$F(t)\leq F(0),\quad\forall t\geq0.$$ \end{propn}
 
 \begin{rem} The result of existence and uniqueness is well-known, cf. e.g. \cite{bre1, bre2, hasem, hab}. For the inequality which is also classical, it can be derived in a more general framework by replacing in the equation the damping term $g( y_t)$ by Yosida's regularization $g_{\lambda}( y_t)$, then multiplying the equation by $ -\Delta y_t$ in the sense of duality between $H^{-1}(\Om) $ and $ H^1_0(\Om)$ and integrating in $s\in[0, t]$. The result then follows by letting $\lambda\to0$.    \end{rem}
 
 \begin{definition} A function $y\in{\cal C}([0,\infty); H^1_0(\Om))\cap{\cal C}^1([0,\infty);L^2(\Om))$ is said to be a {\bf weak  solution}  of \eqre{eg} 
if $y$ is the limit of a sequence $y_n$ of strong solutions of \eqre{eg} in the topology of ${\cal C}([0,\infty); H^1_0(\Om))\cap{\cal C}^1([0,\infty);L^2(\Om)).$

\end{definition}

 In the sequel, $|u|_q$ denotes the 
$L^q(\Om)-$norm of $u$ when $q\geq1$. First, we recall  a result which is valid for any damping term:

\begin{propn}\label{wp} Let $(y^0,y^1)\in  H^1_0(\Om)\times L^2(\Om)$.  System \eqre{eg} has a unique weak solution $y$. Moreover, if $z$ denotes another weak solution of \eqre{eg} corresponding to initial data 
 $(z^0,z^1)$ in $ H^1_0(\Om)\times L^2(\Om)$, then one has the uniqueness-stability inequality:
 \begin{equation}\label{t0}E(y-z;t)=\frac{1}{2}\int_\Om\{| y_t(t,x)- z_t(t,x)|^2+|\nabla y(t,x)-
 \nabla z(t,x)|^2\}\,dx\leq E(y-z;0),\quad\forall t\geq0.\end{equation} \end{propn}
 
 The following new property will turn out to be very useful for our second main result. 
 
 \begin{propn}\label{preg}  If we further assume that $y^0$ lies in $L^p(\Om)$, then the weak solution of \eqre{e1} satisfies $y\in W_{\ell oc}^{1,p}(0,\infty; L^p(\Om))$ with \begin{align}\label{nest1}
 \forall t\geq0, \quad |y(t)|_p\leq2^{p-1\over p}\left(|y_0|_p+c^{- \frac{1}{ p}}t^{p-1\over p}E(0)^{1\over p}\right).\end{align}  \end{propn} 
 \begin{proof} 
To estimate $|y(t)|_p$, first, we note that $y\in C^1([0,\infty;L^2(\Om))$ to derive
$$y(t,x)=y(0,x)+\int_0^ty_t(s,x)\,ds,\quad\forall t>0,\text{ a.e. }x\in\Om.$$ Consequently, it follows from Young inequality
\begin{align}\label{wld050}|y(t,x)|^p\leq 2^{p-1}\left(|y_0(x)|^p+\left(\int_0^t|y_t(s,x)|\,ds\right)^p\right),\quad\forall t>0,\text{ a.e. }x\in\Om.\end{align}
Applying H$\ddot{\text{o}}$lder inequality once more, then integrating over $\Om$, and invoking Fubini's inequality, we get, as a consequence of \eqref{e7} integrated on $(0, t)$ 
\begin{align}\label{wld051}|y(t)|_p^p\leq 2^{p-1}\left(|y_0|_p^p+t^{p-1}\int_0^t|y_t(s)|_p^p\,ds\right)\leq 2^{p-1}\left(|y_0|_p^p+t^{p-1}\frac{E(0)}{c}\right) ,\quad\forall t>0.\end{align} \end{proof} 
 \begin{rem} An estimate similar to \eqref{nest1} may be found in \cite{rtw}. In fact, the authors of \cite{rtw} consider the system, (the actual system in \cite{rtw} includes a source term that is dropped for the sake of simplicity)
\begin{equation}\label{rtw}\left\{
\begin{array}{lll}y_{tt}-\Delta y-\text{div}(|\nabla y_t|^{p-2}\nabla y_t)=0\hbox{ in
}\Om\times(0,\infty)\cr y=0\hbox{ on }\G\times(0,\infty)\cr
y(0)=y^0\in W_0^{1,p}(\Om)\quad y_t(0)=y^1\in L^2(\Om),\end{array}\right.\end{equation}
and, relying on an inequality like \eqref{nest1} and an approximation scheme, they obtain that  the energy 
\begin{equation}\label{rtw1}H(t)={1\over 2}\int_\Omega\{\vert
 y_t(x,t)\vert^2+|\nabla y(x,t)|^2\}\,dx,\quad\forall
t\geq 0,\end{equation} satisfies, for every $p>2$, the logarithmic decay rate
\begin{equation}\label{rtw2}\exists C=C(y^0,y^1,p,\Om)>0:H(t)\leq\frac{C}{(1+\log t)^{p-1}},\quad\forall t\geq1.\end{equation} 

\end{rem}

 We now turn to the statement of our main results. Henceforth, we assume $N\geq3$.

\section{Main results} \begin{thm}\label{stab0}{\bf (A decay property for strong solutions.)} Assume that either $\Om$ is convex, or $\Omega$ has a $C^2$ boundary. Let $p>2$. Let  $(y^0,y^1)\in \left(H^2(\Om)\cap H_0^1(\Om)\right) \times    \left(H^1_0(\Om)\cap L^{2(p-1)}(\Om)\right)$. Assume that $p$ further satisfies the inequalities
   $$p(N-4)\leq 2N<p(N-2).$$ Then the energy of the corresponding strong solution  of \eqref{e1} satisfies the decay estimate:
\begin{equation}\label{nene1}
E(t)\leq{K_1}{\left(1+t\right)^{-{\frac{1}{\mu_{p,N}}}}},\quad\forall t\geq0,\end{equation}where $\mu_{p,N}=\frac{(p-2)}{2}\max\left\{1,\frac{N-4}{4(p-1)}\right\}$, and $K_1=K_1(\Om,N, E(0), p,F(0))$ is a positive constant.  
\end{thm}
 
\begin{thm}\label{stab}{\bf (A logarithmic decay rate under a weak additional assumption on $y^0$.)} Let $p>2$.  Let $y^0\in H_0^1(\Om)\cap L^p(\Om)$ and $y^1\in L^2(\Om)$. Assuming $(N-2)p>2N$, the energy of the corresponding  weak solution of \eqre{e1} given by Theorem \re{wp} satisfies the decay estimate \begin{equation}\label{ene1}
E(t)\leq {K}{(\log(2+t))^{-{1\over \mu_p}}}\quad\forall t\geq0,\end{equation}where $\mu_p=\max\{\frac{p-2}{2}, \frac{1}{p-1}\}$, and $K = K(\Om,N, E(0), p,| y^0|_p)$ is a positive constant.  
\end{thm}
\begin{rem}\label{rem1} The logarithmic decay estimate of the energy in the super-critical case is new, and will be established by relying on a new differential inequality. It  is valid for all weak solutions of System \eqre{e1} under the very light  additional assumption that  the initial displacement lies in $L^p(\Om)$. \end{rem} 
\begin{rem}\label{rem2} As Theorem \re{stab0} shows, strong solutions of \eqref{e1} will still satisfy the decay rate in \eqre{ene0} even if $(N-2)p>2N$, provided that $(N-4)p\leq 2N $  and $\mu_{p,N}=\frac{(p-2)}{2}\max\left\{1,\frac{N-4}{4(p-1)}\right\} = \frac{(p-2)}{2}$. In particular, when $N\ge 5$ and $ p =\frac{ 2N} {N-4} $, we have the energy decay estimates $$\forall t\geq0, \quad E(t)\leq\begin{cases}&K_0(1+t)^{-\frac{N+4}{N-4}},\text{ if }N\ge 12,\\
&K_0(1+t)^{-\frac{N-4}{4}},\text{ if }N\le12.\end{cases} $$  The maximal decay rate of the energy for the maximal value $ p =\frac{ 2N} {N-4} $ is obtained for $N= 12$ and corresponds to $$ E(t) \le K t^{-2}$$ \\
More generally, the decay rate of the energy for $ N\ge 5$  and $ p \le \frac{ 2N} {N-4} $ is given by \eqre{ene0} provided $\mu_{p,N}= \frac{(p-2)}{2}$ which is equivalent to $ p\ge \frac{N}{4}$.   And this is valid for all $p\ge \frac{ 2N} {N-2} $ if  $ N \in \{3, 4\} $  and all $p\in [\frac{ 2N} {N-2},  \frac{ 2N} {N-4}]$ whenever $5\le N\le 10$.   \end{rem}  

\begin{section}{Some technical lemmas}
\begin{lem}\label{ndi} (New differential inequality). Let $\var:[0,\infty)\longrightarrow[0,\infty)$ be a ${\cal C}^1$ strictly increasing function with \begin{equation}\label{l1}\var(0)=0\text{ and }\lim_{t\to\infty}\var(t)=\infty.\end{equation}  Let
$E:[0,\infty[\longrightarrow [0,\infty[$ be a nonincreasing
locally absolutely continuous function such that there exist constants
 $\beta\geq0$ and $A>0$ with   \begin{equation}\label{l2}E'(t)\leq -A\varphi'(t)E(t)^{1+\beta},\text{ a. e. }t\geq0.\end{equation} Then we
have for every $t\geq0$: \begin{equation}\label{l3}E(t)\leq \left\{\begin{array}{ll}&E(0)e^{-A\var(t)},\text{ if }\beta=0,\cr& E(0)\left({1+A\beta E(0)^\beta\var(t)}\right)^{-\frac{1}{
\beta}},\text{ if }\beta>0.\cr \end{array}\right. \end{equation} \end{lem} 
 
 The proof of this lemma is elementary and is left as an easy exercise to the interested reader. The weight function $\var$ is borrowed from Martinez works, e.g. \cite{mar1,mar3}, discussing integral inequalities; the Martinez integral inequalities generalize earlier results of Haraux \cite{hos} and Komornik \cite[Chaps. 8 and 9]{kb} established for $\var(t)=t$. Several generalizations of those integral inequalities exist in the literature e.g. \cite {al,gues}. Note in particular that \eqref{l2} implies the integral inequality
$$\int_S^\infty\var'(t)E(t)^{1+\beta}\,dt\leq (1/A)E(S),\quad\forall S\geq0.$$
 \\  
\begin{lem}\label{gni} (Gagliardo-Nirenberg interpolation inequality) \cite{nir}
 Let $1\leq q\leq s\leq\infty$,
$1\leq r\leq s$, $0\leq k<m<\infty$, where $k$ and $m$ are nonnegative
integers,
and let $\d\in[0,1]$. Let $v\in W^{m,q}(\Om)$.
Further assume
that those parameters satisfy
\begin{equation}k-{N\over s}\leq \d(m-\frac{N}{ q})-\frac{N}{r}(1-\d).\end{equation}
Then
$v\in W^{k,s}(\Om)$, and there exists a positive constant $C=C(k,m,s,q,r,\Om)$ such
that
\begin{equation}\label{gni1}||v||_{W^{k,s}(\Om)}\leq
C||v||_{W^{m,q}(\Om)}^\d|v|_r^{1-\d}.\end{equation}\end{lem} Applying the lemma with $ k= 0$, $m = q = r = 2$ and $s= p$ we find 
 
 \begin{cory}\label{interp} Let $p>2 $ be such that $(N-4)p \le 2N$.  Then $ H^2(\Om) \subset L^p(\Om)$  and 
 \begin{equation}\label{interp1} \forall v \in H^2(\Om), \quad  |v|_{p}\leq
C||v||_{H^2(\Om)}^\d|v|_2^{1-\d}\end{equation}
with $$ \d = \frac{ N(p-2)} {4p} . $$ 
 
 \end{cory}
  \label{tl}

\end{section}
\section{ Proof of Theorem \re{stab0}}
We are going to use the perturbed energy method. For this purpose, let $\v>0$ and $\mu\geq0$ be  constants to be determined later.   Introduce the new energy
$$E_\v(t)=E(t)+\v E^{\mu}\int_\Om yy_t\,dx,\quad \forall t\geq0.$$
Using the Cauchy-Schwarz inequality, one readily checks
\begin{align}\label{eeq}
\left(1-\frac{\v E(0)^\mu}{\lambda}\right)E(t)\leq E_\v(t)\leq \left(1+\frac{\v E(0)^\mu}{\lambda}\right)E(t),\quad\forall t\geq0,\end{align}where $\lambda^2$ is the first eigenvalue of minus Laplacian with Dirichlet boundary conditions.  \\ Thus, provided $\v$ is small enough, the two energies are equivalent. Our task now is to show that $E_\v$ satisfies a classical differential inequality from which we will be able to derive the claimed energy decay estimate.  We recall that we are dealing with the super-critical case: $(N-2)p>2N$. \\ From now on, we divide the proofs into two steps. In the first step, we are going to derive a first estimate for $E_\v'(t)$, without completely estimating the term involving the nonlinear damping. In the second step, we shall complete the estimation of the latter term and complete the proof of Theorem \re{stab0}. 
\vskip.2cm\noindent
{\sl Step 1.} Differentiating $E_\v$, then using the first equation in \eqre{e1}, we find
\begin{align}\label{wld1}
E_\v'(t)&=E'(t) +\v E^{\mu}\int_\Om |y_t|^2\,dx-\v E^{\mu}\int_\Om |\nabla y|^2\,dx-c\v E^{\mu}\int_\Om |y_t|^{p-2}y_ty\,dx\notag\\
&\hskip.2in+\mu\v E'E^{\mu-1}\int_\Om yy_t\,dx \notag\\&
=E'(t) +2\v E^{\mu}\int_\Om |y_t|^2\,dx-\varepsilon E^{\mu}\int_\Om \{|y_t|^2+ |\nabla y|^2\}\,dx-c\v\ E^{\mu}\int_\Om |y_t|^{p-2}y_ty\,dx\\
&\hskip.2in+\mu\v E'E^{\mu-1}\int_\Om yy_t\,dx\notag\\&
=E'(t) -2\v E^{\mu+1}+2 \varepsilon E^{\mu}\int_\Om |y_t|^2\,dx-c\v E^{\mu}\int_\Om |y_t|^{p-2}y_ty\,dx\notag\\
&\hskip.2in+\mu\v E'E^{\mu-1}\int_\Om yy_t\,dx.\text{ a. e. }t>0. \notag\end{align}Now, it remains to estimate each of the integral terms. Henceforth, $C$ denotes a generic positive constant that may depend on $\Om$, $p$ or $N$, but never on the initial data.\\ We start with the last integral. Thanks to the Cauchy-Schwarz and Poincar\'e inequalities, we derive
\begin{align}\label{wld2}
\left|\mu\v E'E^{\mu-1}\int_\Om yy_t\,dx\right|\leq C\v E(0)^\mu|E'(t)|.\end{align}
 The application of H$\ddot{\text{o}}$lder's inequality yields
\begin{align}\label{wld4}
2\v E^{\mu}\int_\Om |y_t|^2\,dx\leq C\v E^{\mu}|y_t|_p^2&\leq C\v E^{\mu}|E'(t)|^{\frac{2}{p}}\notag\\&=C\v E^{\frac{(\mu+1)(p-2)}{p}+\frac{2\mu-(p-2)}{p}}|E'(t)|^{\frac{2}{p}}.\end{align}
Using Young's inequality, and assuming  $ \displaystyle \mu \geq \frac{(p-2)}{2} $, we derive from \eqre{wld4}:
\begin{align}\label{wld40}
2\v E^{\mu}\int_\Om |y_t|^2\,dx\leq \frac{\v(p-2)}{p}E^{\mu+1}+C\v E(0)^{\frac {2\mu-(p-2)}{2}}|E'(t)|.\end{align}
For the last integral, using H$\ddot{\text{o}}$lder's inequality, we find 
\begin{align}\label{wld5}
\left|c\v E^{\mu}\int_\Om |y_t|^{p-2}y_ty\,dx\right|\leq c\v E^{\mu}|y_t(t)|_p^{p-1}|y(t)|_p\leq c\v E^{\mu}|E'(t)|^{\frac{p-1}{p}}|y(t)|_p.\end{align} 
Combining \eqref{wld1}-\eqref{wld5}, we derive 
\begin{align}\label{nwld6}
E_\v'(t)&\leq \left(1-C\v \left(E(0)^\mu+E(0)^{\frac {2\mu-(p-2)}{2}}\right)\right)E'(t)-\frac{\v(p+2)}{p} E^{\mu+1}\notag\\&\hskip.3cm+
 c\v E^{\mu}|E'(t)|^{\frac{p-1}{p}}|y(t)|_p.\end{align} In earlier works, the authors estimated the term $|y(t)|_p$  as:
 \begin{align}\label{wld50}
 |y|_p\leq C|\nabla y(t)|_2.\end{align} Doing that immediately forced them to work in the subcritical or critical framework.\\ In this section, we use the fact that we are dealing with strong solutions. This enables us to get into the supercritical range. To estimate
   the term $|y(t)|_p$, we are going to rely on the Gagliardo-Nirenberg interpolation inequality.
 \vskip.2cm\noindent  
{\sl Step 2. Estimating $|y(t)|_p$ and completing the proof of Theorem \re{stab0}.} Thanks to Corollary \re{interp}, we find
\begin{align}\label{nwld01}|y(t)|_p\leq C|y(t)|_2^{2N-p(N-4)\over4p}||y(t)||_{H^2(\Omega)}^{N(p-2)\over4p}\leq C E(t)^{2N-p(N-4)\over8p}F(0)^{N(p-2)\over8p},\text{ a.e. }t>0.\end{align}Notice that we can use the Gagliardo-Nirenberg interpolation inequality or Sobolev embedding because we are dealing with strong solutions, and so, we are able to exploit the fact that $y(t)$ lies in $H^2(\Om).$ Therefore, we now have
\begin{align}\label{nwld66}
E_\v'(t)&\leq \left(1-C\v \left(E(0)^\mu+E(0)^{\frac {2\mu-(p-2)}{2}}\right)\right)E'(t)-\frac{\v(p+2)}{p} E^{\mu+1}\notag\\&\hskip.3cm+
 C\v E^{\mu}|E'(t)|^{\frac{p-1}{p}}E(t)^{2N-p(N-4)\over8p}F(0)^{N(p-2)\over8p}.\end{align}
Now, we proceed by cases.
\vskip.2cm\noindent
{\bf Case 1: ${\mathbf {p(N-4)= 2N. }}$} This makes sense only if $N>4.$ In this case we have automatically $2N <p(N-2) $,   then \eqref{nwld01} and Young's inequality readily yield,
 under the condition  $\mu(p-1)\geq1$:
\begin{align}\label{nwld2} C\v E^{\mu}|E'(t)|^{p-1\over p}F(0)^{N(p-2)\over8p}&=C\v E^{\mu+1\over p}E^{\mu(p-1)-1\over p} |E'(t)|^{p-1\over p}F(0)^{N(p-2)\over8p}\notag\\&
\leq\frac{\v}{p} E^{\mu+1}+C\v E(0)^{\mu(p-1)-1\over p-1}F(0)^{N(p-2)\over8(p-1)}|E'(t)|.\end{align}
Plugging  this in \eqref{nwld66},  choosing $\v$ small enough and $\mu=\mu_p=\max\{(p-2)/2,1/(p-1)\}$ ,  one finds
\begin{align}\label{nwld3}
E_\v'(t)&\leq -\frac{\v(p+1)}{p} E^{\mu+1}(t).\end{align}
\vskip.2cm\noindent
{\bf Case 2: ${\mathbf {p(N-4)<2N<p(N-2).}}$} Thanks to Young's inequality we have, under the condition $8\mu(p-1)\geq(p-2)(N-4)$:
\begin{align}\label{nwld4} &C\v E^{\mu}|E'(t)|^{p-1\over p}E(t)^{2N-p(N-4)\over8p}F(0)^{N(p-2)\over8p}\notag\\&=C\v E^{\mu+1\over p}E^{8\mu(p-1)-(p-2)(N-4)\over 8p} |E'(t)|^{p-1\over p}F(0)^{N(p-2)\over8p}\\&
\leq\frac{\v}{p} E^{\mu+1}+C\v E(0)^{8\mu(p-1)-(p-2)(N-4)\over 8(p-1)}F(0)^{N(p-2)\over8(p-1)}|E'(t)|.\notag\end{align}
Gathering \eqref{nwld4} and  \eqref{nwld66},  choosing $\v$ small enough and $\mu=\mu_{p,N}=\frac{p-2}{2}\max\left\{1,\frac{N-4}{4(p-1)}\right\}$,  one finds
\begin{align}\label{nwld30}
E_\v'(t)&\leq -\frac{\v(p+1)}{p} E^{\mu+1}(t).\end{align}

Notice that in both cases, we obtain the same differential inequality. Now, thanks to \eqre{eeq}, we have
\begin{align}\label{nwld7}
 -\frac{\v(p+1)}{p} E^{\mu+1}(t)&\leq -\frac{\v(p+1)\lambda^{\mu+1}}{2p\left(\lambda+\v E(0)^\mu\right)^{\mu+1}} E_\v^{\mu+1}(t).\end{align}
The combination of \eqre{nwld30}-\eqre{nwld7} yields for some positive constant $\beta=\beta(\Om,p,N,E(0),F(0)):$ 
$$E_{\v}'(t)\leq -\beta E_{\v}(t)^{\mu+1},\text{ a.e. }t\geq0,$$which is a classical differential inequality.\\
 It then follows from that differential inequality and \eqref{eeq}, the claimed decay estimate
$$E(t)\leq K_0(1+t)^{-\frac{1}{\mu_{p,N}}},\quad\forall t\geq0.$$ Notice that when $p(N-4)=2N$, one has $\mu_{p,N}=\mu_p$. 
This completes the proof of Theorem~\re{stab0}.\qed
\section{ Proof of Theorem \re{stab}} 

Here, we are going to use the perturbed energy method as well. Although we are dealing now with weak solutions, the  formal computations are in fact fully justified since in \cite{h4} it is shown that the energy is absolutely continuous as soon as the damping term is odd, and the formal equality \eqref{e7} is actually satisfied a.e. in $t$.\medskip

Let $\v>0$ and $\mu\geq0$ be  constants to be determined later. Let $\var$ be the function in Lemma \re{ndi}, and further assume that $\var $ is of class ${\cal C}^2$ with $\var'(t)>0$ for every $t$ greater than or equal to zero, and $\var$ is concave.  Introduce the new energy
$$E_\v(t)=E(t)+\v\var'E^{\mu}\int_\Om yy_t\,dx,\quad \forall t\geq0.$$
Using the Cauchy-Schwarz inequality, one readily checks
\begin{align}\label{eeq1}
(1-\frac{\v\var'(0)E(0)^\mu}{\lambda})E(t)\leq E_\v(t)\leq (1+\frac{\v\var'(0)E(0)^\mu}{\lambda})E(t),\quad\forall t\geq0,\end{align}where $\lambda^2$ is, as in the proof of Theorem \re{stab0}, the first eigenvalue of minus Laplacian with Dirichlet boundary conditions.\\ Thus, provided $\v$ is small enough, the two energies are equivalent. Our task now is to show that $E_\v$ satisfies a differential inequality similar to the one in Lemma \re{ndi} with an appropriate function $\var$, and then apply the equivalence and Lemma \re{ndi} to derive the claimed decay estimate. \\ From now on, we divide the proof into two steps as we did for the proof of Theorem \re{stab0}. In the first step, we are going to derive a first estimate for $E_\v'(t)$, without completely estimating the term involving the nonlinear damping. In the second step, we shall complete the estimation of the latter term and complete the proof of Theorem \re{stab}.
\vskip.2cm\noindent
{\sl Step 1.} Differentiating $E_\v$, then using the first equation in \eqre{e1}, we find
\begin{align}\label{wld01}
E_\v'(t)&=E'(t) +\v\var'E^{\mu}\int_\Om |y_t|^2\,dx-\v\var'E^{\mu}\int_\Om |\nabla y|^2\,dx-c\v\var'E^{\mu}\int_\Om |y_t|^{p-2}y_ty\,dx\notag\\
&\hskip.2in+\v\var''E^{\mu}\int_\Om y_ty\,dx+\mu\v\var'E'E^{\mu-1}\int_\Om yy_t\,dx \notag\\&
=E'(t) +2\v\var'E^{\mu}\int_\Om |y_t|^2\,dx-\v\var'E^{\mu}\int_\Om \{|y_t|^2+ |\nabla y|^2\}\,dx-c\v\var'E^{\mu}\int_\Om |y_t|^{p-2}y_ty\,dx\\
&\hskip.2in+\v\var''E^{\mu}\int_\Om y_ty\,dx+\mu\v\var'E'E^{\mu-1}\int_\Om yy_t\,dx\notag\\&
=E'(t) -2\v\var'E^{\mu+1}+2\v\var'E^{\mu}\int_\Om |y_t|^2\,dx-c\v\var'E^{\mu}\int_\Om |y_t|^{p-2}y_ty\,dx\notag\\
&\hskip.2in+\v\var''E^{\mu}\int_\Om y_ty\,dx+\mu\v\var'E'E^{\mu-1}\int_\Om yy_t\,dx.\text{ a. e. }t>0. \notag\end{align}Now, it remains to estimate each of the integral terms. Henceforth, $C$ denotes a generic positive constant that may depend on $\Om$, $p$ or $N$, but never on the initial data.\\ We start with the last integral. Thanks to the Cauchy-Schwarz and Poincar\'e inequalities, we derive
\begin{align}\label{wld02}
\left|\mu\v\var'E'E^{\mu-1}\int_\Om yy_t\,dx\right|\leq C\v E(0)^\mu|E'(t)|.\end{align}
\begin{align}\label{wld03}
\left|\v\var''E^{\mu}\int_\Om y_ty\,dx\right|&\leq(\v/\lambda)|\var''|E^{\mu+1}=-(\v/\lambda)\left(\var'E^{\mu+1}\right)'+(\v(\mu+1)/\lambda)\var'E'E^{\mu}\notag\\&\leq-(\v/\lambda)\left(\var'E^{\mu+1}\right)',\text{ as } |\var''(t)|=-\var''(t),\end{align}where in the equality, we use the concavity of $\var$, and in the last inequality, we use the fact that the energy is nonincreasing.\\ The application of H$\ddot{\text{o}}$lder's inequality yields
\begin{align}\label{wld04}
2\v\var'E^{\mu}\int_\Om |y_t|^2\,dx\leq C\v\var'E^{\mu}|y_t|_p^2&\leq C\v\var'E^{\mu}|E'(t)|^{\frac{2}{p}}\notag\\&=C\v\var'E^{\frac{(\mu+1)(p-2)}{p}+\frac{2\mu-(p-2)}{p}}|E'(t)|^{\frac{2}{p}}.\end{align}
Using Young inequality, and assuming  $ \displaystyle \mu \geq \frac{(p-2)}{2} $, we derive from \eqre{wld4}:
\begin{align}\label{wld040}
2\v\var'E^{\mu}\int_\Om |y_t|^2\,dx\leq \frac{\v(p-2)}{p}\var'E^{\mu+1}+C\v E(0)^{\frac {2\mu-(p-2)}{2}}|E'(t)|.\end{align}
For the last integral, using H$\ddot{\text{o}}$lder's inequality, we find 
\begin{align}\label{wld05}
\left|c\v\var'E^{\mu}\int_\Om |y_t|^{p-2}y_ty\,dx\right|\leq c\v\var'E^{\mu}|y_t(t)|_p^{p-1}|y(t)|_p\leq c\v\var'E^{\mu}|E'(t)|^{\frac{p-1}{p}}|y(t)|_p.\end{align} 
Combining \eqref{wld01}-\eqref{wld05}, we derive 
\begin{align}\label{nwld06}
\left(E_\v+(\v/\lambda)\var'E^{\mu+1}\right)'(t)&\leq \left(1-C\v \left(E(0)^\mu+E(0)^{\frac {2\mu-(p-2)}{2}}\right)\right)E'(t)-\frac{\v(p+2)}{p}\var'E^{\mu+1}\notag\\&\hskip.3cm+
 c\v\var'E^{\mu}|E'(t)|^{\frac{p-1}{p}}|y(t)|_p.\end{align} Estimating the term $|y(t)|_p$ now requires a different tool as we can no longer invoke the Gagliardo-Nirenberg interpolation inequality or Sobolev embedding; neither is available as we are dealing with the supercritical setting and weak solutions.
 \vskip.2cm
\noindent 
{\sl Step 2.}
It is here that we use Proposition \ref{preg}. The combination of \eqre{wld05} and \eqre{nest1} yields
\begin{align}\label{wld052}
c \v\var'E^{\mu}|E'(t)|^{\frac{p-1}{p}}|y(t)|_p\leq C(p)\v|y^0|_p\var'E^{\mu}|E'(t)|^{\frac{p-1}{p}}+C(p)\v (1+t)^{p-1\over p}E(0)^{1\over p}\var'E^{\mu}|E'(t)|^{\frac{p-1}{p}}.\end{align}
Noticing that
$$E^\mu=E^{{\mu+1\over p}+{\mu(p-1)-1\over p}},$$ and using Young inequality,  assuming $\mu(p-1)\geq 1$, we derive from \eqre{wld052} :
\begin{align}\label{wld053}
c\v\var'E^{\mu}|E'(t)|^{\frac{p-1}{p}}|y(t)|_p\leq \frac{2\v}{p}\var'E^{\mu+1}+C_1\v(1+t)\var'\left(|y^0|_p^{p\over p-1}E(0)^{\mu(p-1)-1\over p-1}+E(0)^\mu\right)|E'(t)|.\end{align}
Choosing the function $\var$ such that $(1+t)\var'$ is uniformly bounded, setting $\mu=\mu_p$, and gathering \eqre{nwld06} and \eqre{wld053}, we derive
\begin{align}\label{wld060}
&\left(E_\v+(\v/\lambda)\var'E^{\mu+1}\right)'(t)\notag\\&\leq -\v\var'E^{\mu+1}(t)-|E'(t)|\left(1-C_2\v\left(E(0)^\mu+E(0)^{\frac {2\mu-(p-2)}{2}}+|y^0|_p^{p\over p-1}E(0)^{\mu(p-1)-1\over p-1}\right)\right).\end{align}
Choosing $\v$ small enough, it follows
\begin{align}\label{wld061}
\left(E_\v+(\v/\lambda)\var'E^{\mu+1}\right)'(t)&\leq -\v\var'E^{\mu+1}(t).\end{align}
Observe that, thanks to \eqre{eeq}, we have
\begin{align}\label{wld070}
 E^{\mu+1}(t)&=\frac{1}{2}E^{\mu+1}(t)+\frac{1}{2}E^{\mu+1}(t)\notag\\&\geq \frac{\lambda^{\mu+1}}{2\left(\lambda+\v\var'(0)E(0)^\mu\right)^{\mu+1}}E_\v^{\mu+1}(t) +\frac{1}{2}E^{\mu+1}(t)\\&\geq \delta(\varepsilon) \left(E_\v^{\mu+1}(t)+E^{\mu+1}(t)\right)\notag\end{align} where $\delta(\v) \ge \delta_0>0$ for $\v$ less than $1$. \\
Now, for each $t\geq0$, set 
$$F_{\v,\mu}(t)=E_\v(t)+(\v/\lambda)\var'E^{\mu+1}(t).$$
Then, one checks
\begin{align}\label{wld08}
F_{\v,\mu}(t)^{\mu+1}&\leq 2^\mu\left(E_\v^{\mu+1}(t)+(\v\var'(0)/\lambda)^{\mu+1}E(0)^{\mu(\mu+1)}E^{\mu+1}(t)\right)\notag\\&\leq
 2^\mu\max\{1,(\v\var'(0)/\lambda)^{\mu+1}E(0)^{\mu(\mu+1)}\}\left(E_\v^{\mu+1}(t)+E^{\mu+1}(t)\right).\end{align}
The combination of \eqre{wld061}-\eqre{wld08} yields for some positive constant $\gamma=\gamma(\Om,p,N,|y^0|_p,E(0),\var'(0)):$ 
$$F_{\v,\mu}'(t)\leq -\gamma \var' F_{\v,\mu}(t)^{\mu+1},\text{ a.e. }t\geq0.$$
The application of Lemma \re{ndi} and \eqre{eeq} yield the claimed decay estimate, once we pick, say, $$\var(t)=\log(2+t)-\log(2),\quad\forall t\geq0.$$
This completes the proof of Theorem \re{stab}.\qed

\begin{rem} The proof of Theorem \re{stab} just provided is valid for all weak solutions for which the initial displacement is further required to lie in $L^p(\Omega)$ with $(N-2)p>2N$. The logarithmic decay estimate comes from estimating $|y(t)|_p$ in terms of the initial data and $p$, without any recourse to Sobolev embedding theorems. 
\end{rem}
\section{Some extensions of the method}
The technique devised to prove Theorem \re{stab} is flexible and can be adapted to other second order evolution systems. In the sequel, we discuss some extensions of our result.

\subsection{ A wave equation with two damping mechanisms.} Let $a$ and $b$ be positive constants. Let $2<p\leq q$ be constants as well. Consider the following system 
\begin{equation}\label{s2e1}\left\{
\begin{array}{lll}y_{tt}-\Delta y+a|y_t|^{p-2}y_t+b|y_t|^{q-2}y_t=0\hbox{ in
}(0,\infty)\times\Om\cr y=0\hbox{ on }(0,\infty)\times\G\cr
y(0)=y^0,\quad y_t(0)=y^1\hbox{ in }\Omega.\end{array}\right.\end{equation}
 
 The energy of this system is given as in Section \re{weqs}  by \begin{equation}\label{s2e2}E(t)={1\over 2}\int_\Omega\left\{|
 y_t(t,x)|^2+|\nabla y(t,x)|^2\right\}\,dx,\quad\forall
t\geq 0.\end{equation} The energy $E$ is a nonincreasing function
of the time variable $t$ and its derivative satisfies
\begin{equation}\label{s2e3}E'(t)=-\int_{\Om}\left\{ a|y_t(t,x)|^p+b|y_t(t,x)|^q\right\}\,dx,\quad\forall t\geq0.\end{equation}
The purpose of considering this new system is to find out which of the two damping mechanisms dictates the long time dynamics of the system. 
As far as existence and uniqueness of weak or strong solutions are concerned, for this new system we can apply the results of Section 2 with $g(v):= a|v|^{p-2}v+b|v|^{q-2}v$. We also have 

\begin{cory}  If we further assume that $y^0$ lies in $L^q(\Om)$, then the weak solution of \eqre{s2e1} satisfies $y\in W_{\ell oc}^{1,q}(0,\infty; L^q(\Om))$ with
 \begin{align}\label{nest21} \forall t\geq0,\quad |y(t)|_q\leq2^{q-1\over q}\left(|y_0|_q+b^{- \frac{1}{q}}t^{q-1\over q}E(0)^{1\over q}\right).\end{align} \end{cory} 
   
We now turn to the statement of our decay estimate results.
\begin{thm}\label{stabn01}{\bf (Energy decay of strong solutions.)} Assume that either $\Om$ is convex, or $\Omega$ has a $C^2$ boundary. Let $q>p>2$. Let  $(y^0,y^1)\in \left(H^2(\Om)\cap H_0^1(\Om)\right) \times    \left(H^1_0(\Om)\cap L^{2(q-1)}(\Om)\right)$. Assume that $q$ further satisfies the inequalities
   \begin{align}\label{expc} q(N-4)\leq 2N<q(N-2).\end{align} The energy of the corresponding strong solution  of \eqref{s2e1} satisfies the decay estimate:
\begin{equation}\label{nene1}
E(t)\leq{K_1}{\left(1+t\right)^{-{\frac{1}{\mu_{p,q,N}}}}},\quad\forall t\geq0,\end{equation}where $\mu_{p,q,N}=\max\left\{\frac{p-2}{2},\frac{(q-2)(N-4)}{8(q-1)}\right\}$, and $K_1=K_1(\Om,N, E(0), p,q,F(0))$ is a positive constant.  
\end{thm}
 
\begin{thm}\label{stabn1}{\bf (Energy decay of weak solutions)} Let $q>p>2$ with $(N-2)p>2N$.  Let $y^0\in H_0^1(\Om)\cap L^q(\Om)$ and $y^1\in L^2(\Om)$. The energy of the corresponding  weak solution of \eqre{s2e1} satisfies the decay estimate \begin{equation}\label{nene11}
E(t)\leq {K}{(\log(2+t))^{-{1\over \mu_{p,q}}}}\text{ if }(N-2)p>2N \quad\forall t\geq0,\end{equation} where $\mu_{p,q}=\max\{\frac{p-2}{2}, \frac{1}{q-1}\}$, $K=K(\Om,N, E(0), p,q,| y^0|_p,| y^0|_q)$ some positive constant.  
\end{thm}

\begin{rem}The proofs of Theorems \re{stabn01} and \re{stabn1} go along the lines of the proofs of Theorems \re{stab0} and \re{stab} respectively. We want to draw the reader's attention to the fact that the decay rate in Theorem \re{stabn01} is $O((1+t)^{-{2\over p-2}})$ provided that
\begin{align}\label{ncst}(q-2)(N-4)\leq4(p-2)(q-1).\end{align}In this case, the long time dynamics of the system is dictated by the damping involving the exponent $p$.  When \eqref{ncst} fails, the long time dynamics of the system is dictated by the damping involving the exponent $q$. Notice that if we allow the parameter $p$ to be equal to 2 while keeping $q>2$ with $q(N-2)\leq2N$, then all weak solutions will decay exponentially.  \end{rem}
\begin{rem} We find it helpful to include some details about how we got the constants $\mu_{p,q,N}$ and $\mu_{p,q}$ in Theorem \re{stabn01} and Theorem \re{stabn1} respectively. Notice that the restrictions on $\mu$ comes from estimating the terms $|y_t(t)|_2^2$ and $\int_\Omega \left(a|y_t|^{p-2}y_ty+b |y_t|^{q-2}y_ty\right)\,dx$. \\ We estimate $|y_t(t)|_2$ in terms of $|y_t(t)|_p$. Using H$\ddot{o}$lder's inequality, it follows:
$$|y_t(t)|_2^2\leq C|y_t(t)|_p^2\leq C|E'(t)|^{2\over p}.$$This leads to the restriction $\mu\geq\frac{p-2}{2}$, as the proofs given above show.\\ As for the other term, applying H$\ddot{o}$lder's inequality once more, we find
\begin{align}\label{ncrpe}
\left|\int_\Omega \left(a|y_t|^{p-2}y_ty+b |y_t|^{q-2}y_ty\right)\,dx\right|&\leq C\left(|y_t(t)|_p^{p-1}|y(t)|_p+|y_t(t)|_q^{q-1}|y(t)|_q\right)\notag\\&
\leq C\left(|E'(t)|_p^{p-1\over p}|y(t)|_p+|E'(t)|_q^{q-1\over q}|y(t)|_q\right).
\end{align}We estimate $|y(t)|_q$ as we did in the proofs of Theorems \re{stab0} and \re{stab} respectively. When estimating $|y|_p$, we use two different approaches; for Theorem \re {stabn01}, we estimate $|y(t)|_p$ as we did in the proof of  Theorems \re{stab0}, and for Theorem \re{stabn1}, we estimate $|y(t)|_p$ by interpolation as follows
$$|y(t)|_p\leq |y(t)|_2^{2(q-p)\over p(q-2)} |y(t)|_q^{q(p-2)\over p(q-2)}.$$ The interpolation inequality is established by starting with
$$\int_\Omega|y(t,x)|^p\,dx=\int_\Omega|y(t,x)|^{\tau p}|y(t,x)|^{(1-\tau) p}\,dx$$for some $\tau$ in $(0,1)$ to be determined. Next, using H$\ddot{o}$lder's inequality, we find
$$\int_\Omega|y(t,x)|^{\tau p}|y(t,x)|^{(1-\tau) p}\,dx\leq |y(t)|_2^{\tau p}|y(t)|_{2(1-\tau)p\over2-\tau p}^{p(1-\tau )},$$ assuming $\tau<2/p$. Now, setting 
${2(1-\tau)p\over2-\tau p}=q$, we readily derive $\tau=2(q-p)/p(q-2)$. Hence the claimed estimate.\\ This interpolation leads to the restriction $\mu\geq \frac{p-2}{(q-2)(p-1)}$: indeed we have
\begin{align}\label{nstep} E^\mu|y|_p|E'|^{p-1\over p}&\leq CE^{{\mu+1\over p}+\frac{q-p+\mu p(q-2)-(\mu+1)(q-2)}{p(q-2)}}|y|_q^{q(p-2)\over p(q-2)}|E'|^{p-1\over p}\notag\\&=CE^{{\mu+1\over p}}E^{\frac{\mu (q-2)(p-1)-(p-2)}{p(q-2)}}|y|_q^{q(p-2)\over p(q-2)}|E'|^{p-1\over p}\\&\leq CE^{{\mu+1\over p}}E(0)^{\frac{\mu (q-2)(p-1)-(p-2)}{p(q-2)}}|y|_q^{q(p-2)\over p(q-2)}|E'|^{p-1\over p},\end{align}where the last inequality holds under the stated restriction on $\mu$.\\
Notice that estimating $|y|_q$ in Theorem \re{stabn1} leads to $\mu\geq\frac{1}{q-1}$. Hence the claimed expression of $\mu_{p,q}$.
\end{rem}
\subsection{A hinged plate equation with a nonlinear frictional damping.} Let $c>0$ be a constant and consider the
following damped Euler-Bernoulli
equation:
\begin{equation}\label{hse1}\left\{ \begin{array}{lll}&y_{tt}+\Delta^2 y+ c| y_t|^{p-2} y_t   =
0\hbox{  in }
(0,\infty)\times\Omega \cr &y={\Delta y}   =  0      \hbox{ on }
(0,\infty)\times{\Gamma}\cr &{y}(0)   = y^0,
\quad {y_t}(0)   =  y^1     \hbox{ in }        \Omega.\cr
\end{array}\right.\end{equation}
 The energy of this system is given by
$$E(t)=\frac{1}{2}\int_\Om\{|y_t(t,x)|^2+|\Delta y(t,x)|^2\}\,dx,\quad\forall t\geq0,$$and it is a nonincreasing function of the time variable $t$ as we have 
$$E'(t)=-c\int_\Om | y_t(t,x)|^{p}\,dx,\text{ for a.e. }t>0.$$
For this new system, strong and weak solutions are defined in an analogous way as above. We have the following energy decay estimate results:
\begin{thm}\label{stabp0}{\bf (A decay property for strong solutions.)} Assume that either $\Om$ is convex, or $\Omega$ has a $C^4$ boundary. Let $p>2$. Let  $(y^0,y^1)\in \left(H^3(\Om)\cap H_0^1(\Om)\right) \times    \left(H^1_0(\Om)\cap L^{2(p-1)}(\Om)\right)$. Assume that $p$ further satisfies the inequalities
   $$p(N-6)\leq 2N<p(N-4).$$ Then the energy of the corresponding strong solution  of \eqref{hse1} satisfies the decay estimate:
\begin{equation}\label{pene1}
E(t)\leq{K_1}{\left(1+t\right)^{-{\frac{1}{\mu_{p,N}}}}},\quad\forall t\geq0,\end{equation}where $\mu_{p,N}=\frac{(p-2)}{2}\max\left\{1,\frac{N-6}{6(p-1)}\right\}$, and $K_1=K_1(\Om,N, E(0), p,F(0))$ is a positive constant.  The function $F$ is defined by 
$$F(t)=\frac{1}{2}\int_\Om\{|\nabla y_t(t,x)|^2+|\nabla\Delta y(t,x)|^2\}\,dx,\quad\forall t\geq0,$$ and it can be shown that it is nonincreasing.
\end{thm}
 
\begin{thm}\label{stabp1}{\bf (A logarithmic decay rate under a weak additional assumption on $y^0$.)} Let $p>2$.  Let $y^0\in \left(H^2(\Om)\cap H_0^1(\Om)\cap L^p(\Om)\right)$ and $y^1\in L^2(\Om)$. Assuming $(N-4)p>2N$, the energy of the corresponding  weak solution of \eqre{hse1}  satisfies the decay estimate \begin{equation}\label{pene01}
E(t)\leq {K}{(\log(2+t))^{-{1\over \mu_p}}}\quad\forall t\geq0,\end{equation}where $\mu_p=\max\{\frac{p-2}{2}, \frac{1}{p-1}\}$, and $K = K(\Om,N, E(0), p,| y^0|_p)$ is a positive constant.  
\end{thm}
\subsection{A plate equation with a nonlinear structural damping.} A few years ago the second author of this work considered the
following damped Euler-Bernoulli
equation \cite{tebpnd}:
\begin{equation}\label{se1}\left\{ \begin{array}{lll}&y_{tt}+\Delta^2 y- \text{div}(a|\nabla y_t|^{p-2}\nabla y_t)   =
0\hbox{  in }
(0,\infty)\times\Omega \cr &y={\p y\over\p\nu}   =  0      \hbox{ on }
(0,\infty)\times{\Gamma}\cr &{y}(0)   = y^0,
\quad {y'}(0)   =  y^1      \hbox{ in }        \Omega.\cr
\end{array}\right.\end{equation}
 System \eqref{se1} corresponds to the clamped plate equation with structural damping
\cite{cru} when $a\equiv1$, $p=2$, and $N=2$; thus \eqref{se1} is a nonlinear
generalization of the mathematical model proposed by Russell and Chen in their
work. \\ The energy of this system is given by
$$E(t)=\frac{1}{2}\int_\Om\{|y_t(t,x)|^2+|\Delta y(t,x)|^2\,dx,\quad\forall t\geq0,$$and it is a nonincreasing function of the time variable $t$ as we have 
$$E'(t)=-\int_\Om a(x)|\nabla y_t(t,x)|^{p}\,dx,\text{ for a.e. }t>0.$$
The damping was then localized in a convenient subdomain, and several stabilization results were established, \cite{tebpnd}. The following result was also established in the case of a globally distributed damping:
\begin{thm}\cite{tebpnd} Suppose that $a\equiv1$ in $\Om$. Let $p>2$ further satisfy:
$$(N-2)p<3N-4,\hbox{ and }(N-2)p\leq 2N.$$ Then every weak solution of
\eqref{se1} satisfies
$$E(t)\leq K(E(0))(1+t)^{-{2\over p-2}},\quad\forall
t\geq0,$$where $K$ is a positive constant depending on the
initial data as indicated, and also depends on the parameters of the system.

\end{thm}
For this new system, if one drops the restrictions on the parameter $p$, namely choosing $p$ with $p(N-2)>2N$, and $y^0\in H_0^2(\Om)\cap W^{1,p}(\Om)$, $y^1\in L^2(\Om)$, then using the technique devised to prove Theorem \re{stab}, one can show that the energy of the corresponding weak solution of \eqref{se1} satisfies
\begin{equation}\label{se2}
E(t)\leq
{K}{(\log(2+t))^{-{1\over \mu_p}}},\quad\forall t\geq0,\end{equation}where $\mu_p=\max\{\frac{p-2}{2}, \frac{1}{p-1}\}$, and $K=K(\Om,N, E(0), p,| y^0|_p)$ is a positive constant.

\end{document}